\documentclass[runningheads,a4paper]{llncs}
\usepackage{amssymb}
\usepackage{graphicx}
\usepackage{xcolor}

\usepackage{url}
\urldef{\mailsa}\path|armario@us.es|
\urldef{\mailsb}\path|dane.flannery@nuigalway.ie|

\newcommand{\keywords}[1]{\par\addvspace\baselineskip
\noindent\keywordname\enspace\ignorespaces#1}

\newcommand{\mod}{\hspace{3pt} \mbox{mod}\hspace{3pt}}

\newcommand{\abk}{\allowbreak}

\begin{document}

\mainmatter  
\title{On quasi-orthogonal cocycles}

\titlerunning{Quasi-orthogonal cocycles}

\author{J.~A.~Armario \and D.~L.~Flannery}

\authorrunning{Quasi-orthogonal cocycles}

\institute{Departamento de Matem\'atica Aplicada I, 
Universidad de Sevilla,\\ 
Avda. Reina Mercedes s/n, 41012 Sevilla, Spain\\
\mailsa 
\and 
School of Mathematics, Statistics and Applied Mathematics,\\ 
National University of Ireland, Galway, Ireland\\ 
\mailsb}

\maketitle

\begin{abstract}
We introduce the notion of \emph{quasi-orthogonal cocycle}. 
This is motivated in part by the maximal determinant problem 
for square $\{\pm 1\}$-matrices of size congruent to $2$ 
modulo $4$. Quasi-orthogonal 
cocycles are analogous to the \emph{orthogonal cocycles} of
algebraic design theory. 
Equivalences with new and known combinatorial objects afforded 
by this analogy, such as quasi-Hadamard groups, relative 
quasi-difference sets, and certain partially balanced incomplete 
block designs, are proved.
\keywords{Cocycle, (quasi-)orthogonal, (quasi-)Hadamard group, 
difference set, block design.}
\end{abstract}

\medskip

\section{Introduction}
In the early 1990s, de Launey and Horadam discovered 
\emph{cocyclic development of pairwise combinatorial designs}. 
Their discovery opened up a new area in design theory, 
emphasizing algebraic methods drawn mainly from group theory
and cohomology. See \cite{DF11,Hor07} for comprehensive expositions.

Let $G$ and $U$ be finite groups, with $U$ abelian. A map 
$\psi: G\times G\rightarrow U$ such that
\begin{equation}
\label{CocycleIdentity}
\psi(g,h)\psi(gh,k)=\psi(g,hk)\psi(h,k) \quad\forall \, g,h,k\in G
\end{equation}
is a \emph{cocycle} (\emph{over $G$}, \emph{with coefficients 
in $U$}). 
We may assume that $\psi$ is normalized, i.e., $\psi(1,1)=1$. 
For any (normalized) map $\phi: G\rightarrow U$, the cocycle $\partial\phi$  
defined by 
$\partial\phi(g,h)=\phi(g)^{-1}\phi(h)^{-1}\phi(gh)$ is
a {\em coboundary}.
The set of cocycles $\psi:\abk G\times G\rightarrow U$ forms an 
abelian group $Z^2(G,U)$ under pointwise multiplication.  
Factoring out $Z^2(G,U)$ by the subgroup of coboundaries gives 
the \emph{second cohomology group of $G$ with coefficients in $U$}, denoted
$H^2(G,U)$.

Each cocycle $\psi\in Z^2(G,U)$ is displayed as a 
{\em cocyclic matrix} $M_\psi$: under some indexings of the rows and 
columns by $G$, $M_\psi$ has entry $\psi(g,h)$ in 
position $(g,h)$. Our principal focus in this paper is the case 
$U= \langle -1 \rangle\cong \mathbb{Z}_2$.
We say that  $\psi$ is {\em orthogonal} if $M_\psi$ is a Hadamard 
matrix, i.e., $M_\psi M_\psi^\top=nI_n$ where $n=|G|$ is 
necessarily $1$, $2$, or a multiple of $4$.

The paper \cite{DFH00} describes explicit links between orthogonal 
cocycles and other combinatorial objects.
For example, we can use an orthogonal cocycle 
to construct a relative difference set with 
forbidden subgroup $\mathbb{Z}_2$ in a central 
extension of $\mathbb{Z}_2$ by $G$, and vice versa. 
Such extensions, known as {\em Hadamard groups}, were studied by 
Ito in a series of papers beginning with  \cite{Ito94}. Their 
equivalence with 
cocyclic Hadamard matrices was demonstrated in \cite{Fla97}. 
There is a further equivalence 
with class regular group divisible designs on which the
Hadamard group acts as a regular group of automorphisms. 
Techniques and results have been translated fruitfully between
the different contexts. 

Recent work on cocycles over groups of even 
order not divisible by $4$ has been motivated by 
the role that such cocycles play in the 
maximal determinant problem for binary matrices~\cite{AAFG12a,AAFG15}. 
The present paper examines the existence, 
classification, and combinatorics of such cocycles under an appropriate version of 
orthogonality---a modified balance condition 
on rows and columns. 
We prove versions of the equivalences in \cite{DFH00}. 
The paper is a launching point for investigation of all these new 
algebraic and combinatorial ideas.

Throughout, $I$ denotes an identity matrix and $J$ a square all-$1$s 
matrix. The Kronecker product of $A=[a_{i,j}]$ and $B$
is  $A\otimes B:= [a_{i,j}B]$.  
Given a matrix $M=[m_{i,j}]$, we write $\mbox{abs}(M)$ for 
$[ \hspace{.8pt} |m_{i,j} |\hspace{.85pt} ]$.

\section{Quasi-orthogonal cocycles}

A Hadamard matrix with normalized first row (each entry equal 
to $1$) has zero row sums everywhere else. The same statement 
with `row' replaced by `column' is also true. As it happens,
this constraint on rows and columns characterizes
the cocyclic matrices that are Hadamard:
$\psi\in Z^2(G,\mathbb{Z}_2)$ 
is orthogonal if and only if $|\{h\in G \mid  \psi(g,h)=1\}|=|G|/2$
(equivalently, $|\{h \in G \mid \psi(h,g)=1\}|=|G|/2$) for each 
$g\in G\setminus\{1\}$.

Let $M=[m_{i,j}]$ be an $n\times n$ $(-1,1)$-matrix with 
normalized first row. The \emph{row excess} 
\[
RE(M)=\sum_{i=2}^n \Big| \sum_{j=1}^n m_{i,j} \Big|
\]
measures how close the row sums of $M$ are to zero. 
Assuming that $n\equiv 0 \mod 4$, a cocycle $\psi$ over 
a group $G$ 
of order $n$ is orthogonal if and only if $RE(M_\psi)=0$. 
We will give a comparable minimality condition 
on row excess for cocyclic matrices of orders 
$n \equiv 2 \mod 4$.

Denote the Grammian $MM^\top$ by $\mathrm{Gr}(M)$.
Fix an ordering $g_1=1, \allowbreak g_2,\allowbreak 
\ldots,g_{n}$ of $G$ to index $M_\psi=[\psi(g_i,g_j)]$. 
\begin{lemma}[{\cite[Lemma~6.6]{Hor07}}] \label{prop1} 
$\mathrm{Gr}(M_\psi)$ has  $(i,j)$th entry
\[
\psi(g_ig_j^{-1},g_j)\, 
\sum_{g\in G}\psi(g_ig_j^{-1},g).
\]
\end{lemma}
\begin{proof}
Manipulations with the cocycle identity (\ref{CocycleIdentity}). 
\hfill $\Box$
\end{proof}

Unless stated otherwise, henceforth $G$ is a group of order 
$4t+2\geq 6$. Thus $G$ has a (normal) splitting subgroup 
of order $2t+1$.

Each row of a $(-1,1)$-matrix may be designated as even or odd, 
according to the parity of the number of $1$s that it contains.
Note that rows of different parity cannot occur in a Hadamard matrix 
of order $>2$.
\begin{proposition}[cf.~{\cite[Proposition 2]{AAFG12a}}] 
\label{prophadi}
Let $M$ be a cocyclic matrix with indexing group $G$ and let $e$ 
be the number of its even rows. Then
\begin{itemize}
\item[{\rm (i)}] $e= 4t+2$ or $2t+1;$ so $RE(M)\geq 4t$.
\item[{\rm (ii)}] $RE(M)= 4t$ if and only if 
\begin{equation}\label{formachuli}
\mathrm{abs}(\mathrm{Gr}(M))= 
\left[\renewcommand{\arraycolsep}{.12cm}\begin{array}{cc}
{\small 4tI+2J} &  0 \\
\vspace{-10pt}  & \\
0  & {\small 4tI+2J}
\end{array}\right]
\end{equation}
up to row permutation.
\end{itemize}
\end{proposition}
\begin{proof}
Two rows of different (respectively, the same) parity in $M$ have 
inner product $0$ (respectively, $2$) modulo $4$. 
Hence $2e(4t+2-e)$ entries of $\mathrm{Gr}(M)$  are congruent to 
$0$ modulo $4$.  On the other hand, because a
row of $M$ sums to $0$ modulo $4$ if and only if it is odd,
Lemma~\ref{prop1} implies that each row of $\mathrm{Gr}(M)$ has 
precisely $4t+2-e$ entries congruent to
$0$ modulo $4$. Now (i) is apparent.

If $RE(M)= 4t$  then we get the Grammian (\ref{formachuli}) after  
permuting rows of $M$ so that the first 
$2t+1$ rows are even. Conversely, if (\ref{formachuli})
holds then $e=2t+1$, the only non-initial rows of $M$ with 
non-zero sum are rows $2,\ldots , 2t+1$, and that sum is $\pm 2$.
\hfill $\Box$
\end{proof}

Combined with our earlier observation that full orthogonality of 
a cocycle $\psi$ is the same as $RE(M_\psi)$ being minimal, 
Proposition~\ref{prophadi} suggests the following.
\begin{definition}\label{REDefinition}
{\em $\psi\in Z^2(G,\mathbb{Z}_2)$ is {\em quasi-orthogonal} if 
$RE(M_\psi)= 4t$.}
\end{definition}
 
The next result, a useful characterization 
of quasi-orthogonality,
essentially just rephrases Proposition \ref{prophadi}~(ii). 
\begin{lemma}\label{lemmaquasiortho}
For $\psi\in Z^2(G,\mathbb{Z}_2)$, let
\[
X_1=\{g\in G\setminus \{1\} \mid
{\textstyle \sum_{h\in G}}\psi(g,h)=\pm 2\}
\]
and
\[
X_2=\{g\in G\setminus \{1\} \mid  
{\textstyle \sum_{h\in G}}\psi(g,h)=0\}.
\]
Then $\psi$ is quasi-orthogonal if and only if 
$|X_1|=2t$ and $|X_2|=2t+1$.
\end{lemma}

We record some facts about the existence of quasi-orthogonal 
cocycles.
\begin{proposition}\label{CoboundaryNotQO}
No coboundary is quasi-orthogonal.
\end{proposition}
\begin{proof}
Observe that $M=M_{\partial\phi}$ is Hadamard 
equivalent to the group-developed matrix $N=[\phi(gh)]_{gh}$. 
Thus, if $\partial\phi$ is quasi-orthogonal  and
$\mbox{abs}(\mathrm{Gr}(M))$ has the form 
(\ref{formachuli}), then $\mathrm{abs}(\mathrm{Gr}(N))$ does as well.
It follows that $J\mathrm{Gr}(N) \equiv 2J \mod 4$.
Also $J\mathrm{Gr}(N) = k^2J$ where $k$ denotes the 
constant row and column sum of $N$. 
But of course $k^2\not \equiv 2 \mod 4$.\hfill $\Box$
\end{proof}
\begin{remark}
Indeed, every row of $M_{\partial \phi}$ is even; 
from which it is immediate that $\partial \phi$ cannot be 
quasi-orthogonal.
\end{remark}
\begin{remark}\label{OrthCobExist}
Orthogonal coboundaries exist 
(in square orders).
\end{remark}

After carrying out exhaustive searches using 
{\sc Magma}~\cite{Magma}, we found quasi-orthogonal 
cocycles over every group of order $4t+2\leq 42$.
\begin{example}(R.~Egan.)
Take any Hadamard matrix with circulant core and let 
$A$ be the normalized core. Then
${\tiny
\left[ \renewcommand{\arraycolsep}{.05cm}\begin{array}{rr} 
1& 1 \\
1 & - 1
\end{array} \right]}\otimes A$ 
displays a quasi-orthogonal cocycle.
\end{example}

\vspace{2.5pt}

By contrast, groups over which there 
are no cocyclic Hadamard matrices start appearing at 
order $8$. Also, from order $24$ onwards there exist 
Hadamard matrices that are not cocyclic: 
see \cite[Table~1]{OR11}.

A cocyclic matrix of order $4t+2$ whose determinant has absolute
value attaining the 
Ehlich--Wojtas bound $2(4t+1)(4t)^{2t}$ must be 
quasi-orthogonal~\cite[Proposition 3]{AAFG12a}.
Examples of quasi-orthogonal cocycles are thereby available 
in \cite{AAFG12a,AAFG12b}. 
So far, at every order $4t + 2$ such that $4t + 1$
is the sum of two squares, we have always found a $G$ over which 
some quasi-orthogonal cocycle has a matrix attaining the
Ehlich--Wojtas bound.

Cohomological equivalence of cocycles does not preserve 
orthogonality nor quasi-orthogonality.
However, both properties are preserved by a certain
`shift action' on each cocycle class. 
For $a\in G$, this action 
maps $\psi \in Z^2(G,\mathbb{Z}_2)$  to 
$\psi \hspace{1pt} a:=\psi \partial\hspace{.5pt} \psi_a$, 
where $\psi_a(x)= \psi(a,x)$; 
see \cite[Definition~3.3]{Hor00}.
By Lemma~\ref{prop1},
the sum $\sum_{h\in G}\psi(a,h)\psi(ag,h)$ of row $g\neq 1$ in 
$M_{\psi\hspace{.5pt} a}$ is either 
a non-initial row sum of $M_\psi$, or the negation of one.
Hence, by Lemma~\ref{lemmaquasiortho}, $\psi\hspace{1pt} a$ is 
quasi-orthogonal if and only if $\psi$ is too (this is the same 
argument as the one in the proof of \cite[Lemma 4.9]{Hor00} for 
orthogonal cocycles).

\section{Quasi-Hadamard groups}

A group $E$ of order $8t$ is a \emph{Hadamard group} if
it contains a \emph{Hadamard subset}: a 
transversal $T$ for the cosets of a central subgroup 
$Z \cong \mathbb{Z}_2$ such that 
$|T\cap xT|=\abk 2t$ for all $x\in E\setminus Z$
(in fact $x\in T\setminus Z	$ suffices; 
cf.~Remark~\ref{mirror} below). 
These definitions are due to Ito~\cite{Ito94}. He showed that 
the dicyclic group
\[
Q_{8t}=\langle a,b\mid
a^{2t}=b^2,\,
b^4 =1, \,
b^{-1}ab=a^{-1}\rangle 
\]
is a Hadamard group whenever $2t-1$ or $4t-1$ is 
a prime power~\cite{Ito97},
and conjectured that $Q_{8t}$ is always a Hadamard group. 
In \cite{Fla97}, Hadamard groups are shown to coincide with
cocyclic Hadamard matrices, and Ito's conjecture is verified for 
$t\leq 11$. 
Schmidt~\cite{ScIto} later extended the verification up to $t=46$.

We now define the analog of Hadamard group.
\begin{definition}\label{quasihadamardgroup}
\emph{Let $E$ be a group of order $8t+4\geq 12$ with central subgroup 
$Z\cong \mathbb{Z}_2$. We say that $E$ is a {\em quasi-Hadamard group} 
if there exists a transversal $T$ for $Z$ in $E$ 
containing a subset $S\subset T\setminus Z$ of size $2t+1$ such
that}
\begin{equation}\label{qHGroupConditions}
|T\cap xT|=\left\{\begin{array}{ll}
\ 2t+1 &  \hspace{15pt} x \in S \\
\ 2t \ \mathrm{ or } \ 2t+2 & \hspace{15pt} x\in 
T\setminus (S \cup Z).
\end{array}\right.
\end{equation}
\end{definition}
\begin{remark}\label{mirror}
For any $x\in E$ and the non-trivial element $z$ of 
$Z$, $|T\cap xT|= n$ if and only if $|T\cap xzT|=\abk 4t+2-n$.
\end{remark}

We call the transversal $T$ in Definition~\ref{quasihadamardgroup} 
a {\em quasi-Hadamard subset} of $E$. 
It may be assumed that $1\in T$.

Given a group $G$ and $\psi\in Z^2(G, \langle -1\rangle)$, 
denote by $E_\psi$ the canonical 
central extension of $\langle -1\rangle$ by 
$G$; this has elements $\{(\pm 1,g) \mid g\in G \}$ and 
multiplication $(u,g) \;(v,h)=(uv\hspace{.5pt} \psi(g,h),gh)$.
In the other direction, suppose that $E$ is a finite group with 
normalized transversal $T$ for a central subgroup 
$\langle -1\rangle \cong \mathbb{Z}_2$. Put
$G= E/\langle -1\rangle$ and 
$\sigma (t\langle -1\rangle) = t$ for $t\in T$. The map
$\psi_T:G\times G\rightarrow \langle -1\rangle$ defined by
$\psi_T(g,h)=\sigma(g)\sigma(h)\sigma(gh)^{-1}$ is a cocycle; 
furthermore, $E_{\psi_{T}}\cong E$.
\begin{theorem}[cf.~{\cite[Propositions~3.3 and 3.4] {Fla97}}]
\label{orthococyquasihadamard}\
\begin{itemize}
\item[{\rm (i)}]
If $\psi$ is quasi-orthogonal then 
$T=\{(1,g)\mid g \in G\}$ is a quasi-Hadamard subset 
of $E_\psi$.
\item[{\rm (ii)}] If $E$ has quasi-Hadamard subset $T$ 
then $\psi_T$ is quasi-orthogonal.
\end{itemize}
\end{theorem}
\begin{proof}
(i) \, 
For each $x=(u,g)\in E_\psi$, $|T\cap xT |$ counts 
the number of $h\in G$ such that $\psi(g,h) = u$. Hence
\[
|T\cap xT |=
\left\{
\begin{array}
{cl}
2t & \hspace{15pt} x 
\in \{1\}\times X_{1,-}\;\cup\;\{-1\}\times X_{1,+}\\
2t+1 &  \hspace{15pt} x  
\in \{-1,1\} \times X_2 \\
2t+2 & \hspace{15pt} x  
\in \{1\}\times X_{1,+}\;\cup\;\{-1\}\times X_{1,-}
\end{array}
\right.
\]
where $X_{1,\pm}= 
\{g\in G\setminus \{1\}\mid  \sum_{h\in G}\psi(g,h)=\pm 2\}$, 
and $X_2$, $X_1=X_{1,+}\cup \abk X_{1,-}$ are as in 
Lemma~\ref{lemmaquasiortho}.  So  (\ref{qHGroupConditions}) 
holds with $S = \{1\}\times X_2$.

(ii) \,
Let $S$ be as in Definition~\ref{quasihadamardgroup}.
Since $\psi_T(g,h)=1 \Leftrightarrow \sigma(g)\sigma(h)\in T$, 
the number of $h\in G$ such that $\psi_T(g,h)=1$ 
for fixed $g\neq 1$ is
$|T\cap \sigma(g)^{-1}T|=|\sigma(g)T\cap T|$, which equals
$2t+1$ if $\sigma(g)\in S$ and $2t$ or $2t+2$ otherwise,
by (\ref{qHGroupConditions}).
Now this part follows from Lemma~\ref{lemmaquasiortho}, with
$X_1 = \{g\in G\setminus\{1\} \mid \sigma(g) \not \in S \}$ and 
$X_2= \{ g\in G \mid \sigma(g) \in S\}$.\hfill $\Box$
\end{proof}

Theorem~\ref{orthococyquasihadamard} shows that quasi-orthogonal 
cocycle and quasi-Hadamard group are essentially the same 
concept.

Let $D_{4t+2}$ denote the dihedral group of order $4t+2$.
If $\psi\in Z^2(D_{4t+2},\mathbb{Z}_2)$ is not a coboundary then
$E_{\psi}$ is the group  $Q_{8t+4}$ 
with presentation 
\[
\langle a, b \; | \; a^{2t+1}=b^{2},\;b^4=1,\;
b^{-1}ab=a^{-1}\rangle .
\] 
Note that $Q_{8t+4}\cong C_{2t+1}\rtimes C_4$.
We propose an analog of Ito's conjecture 
that the cocycle class in $H^2(D_{4t},\mathbb{Z}_2)$ labeled 
$(A,B,K)=(1,-1,-1)$ in \cite{Fla97} always has  orthogonal 
elements; equivalently, $Q_{8t}$ is always a Hadamard group.
\begin{conjecture}\label{QuasiItoConjecture} 
$Q_{8t+4}$ is a quasi-Hadamard group for all $t\geq 1$.
\end{conjecture}

Conjecture~\ref{QuasiItoConjecture}
has been verified up to $t=10$, by our computer search for 
quasi-orthogonal cocycles.
Actually, for fixed isomorphism type of $G$, there are very few 
possible isomorphism types of quasi-Hadamard groups arising from
cocycles over $G$.
\begin{lemma}\label{SmallSecondCohomology}
$H^2(G, \mathbb{Z}_2) \cong \mathbb{Z}_2$.
\end{lemma}
\begin{proof} 
First, $H_2(G)\cong H_2(N)$ 
where $N\leq G$ is a splitting subgroup of index $2$
(see, e.g., \cite[2.2.6, p.~35]{Kar87}). Then
$H^2(G,\mathbb{Z}_2) \cong 
\mathrm{Ext}(G/G',\mathbb{Z}_2) \cong \mathbb{Z}_2$ 
by the Universal Coefficient Theorem, because 
$|H_2(N)|$ is odd.\hfill $\Box$
\end{proof}

Lemma~\ref{SmallSecondCohomology} and 
Proposition~\ref{CoboundaryNotQO} imply
\begin{corollary}
For each $t\geq 1$ and fixed $G$, there are at most two non-isomorphic 
quasi-Hadamard groups arising from elements of $Z^2(G,\mathbb{Z}_2)$.
\end{corollary}
\begin{remark}
For example, if $G$ is cyclic or dihedral then a quasi-Hadamard
group must be isomorphic to
$C_{8t+4}$ or $Q_{8t+4}$.
\end{remark}
\begin{remark}
While all quasi-Hadamard groups are solvable, there 
exist non-solvable Hadamard groups. 
\end{remark}

Besides Conjecture~\ref{QuasiItoConjecture}, Ito proved 
two results for Hadamard groups that have had important 
consequences for the existence question in the theory of cocyclic 
Hadamard matrices; see \cite[Corollaries~15.6.2 
and 15.6.5, pp.~184--185]{DF11}.
We quote these for comparison with the less interesting 
situation for quasi-Hadamard groups
(each of which has Sylow $2$-subgroup $C_4$). 
\begin{theorem}
Suppose that $H$ is a cocyclic Hadamard matrix of order
greater than $2$ over a group $G$ with cyclic Sylow $2$-subgroups. 
Then $H$ is group-developed over $G$; i.e.,
the corresponding Hadamard group does not have cyclic 
Sylow $2$-subgroups.
\end{theorem}
\begin{theorem}
No Hadamard group has a dihedral Sylow $2$-subgroup.
\end{theorem}

\section{Relative quasi-difference sets }
Let $E$ be a group of order $vm$ with normal subgroup $N$ of 
order $m$. 
A \emph{relative $(v,m,k,\lambda)$-difference set in $E$ with 
forbidden subgroup $N$} 
is a $k$-subset $R$ of a transversal for $N$ in $E$, such 
that if $x\in E\setminus N$ then $x=r_1r_{2}^{-1}$ for exactly 
$\lambda$ pairs $r_1$, $r_2\in R$.
The last condition may be rewritten as
\begin{equation}\label{RDSCondition}
|R\cap xR|=\lambda \quad \forall\, x\in E\setminus N.
\end{equation}                                                                                                                                       
An important special case in which $k=v$ is the following.
\begin{proposition}[{\cite[Corollary 2.5]{DFH00}}]
\label{RDSOC}
Let $|G|=4t$. A cocycle
$\psi\in \abk Z^2(G, \langle -1\rangle)$ 
is orthogonal if and only if $\{(1,g) \mid  g\in G\}$ is a relative 
$(4t,2,4t,2t)$-difference set in $E_\psi$ with forbidden subgroup 
$\langle (-1,1)\rangle$.
\end{proposition}

In other words, a relative 
$(4t,2,4t,2t)$-difference set is a Hadamard subset of a 
Hadamard group, and vice versa.
However, when $t$ is odd, Hiramine~\cite{Hir03} proved that there 
are no relative $(2t,2,2t,t)$-difference sets. 
So we need an analog of relative difference set for quasi-Hadamard 
groups.
\begin{definition}
\emph{Let $E$ a group of order $8t+4$, and $Z$ a normal (hence 
central) subgroup of order $2$.  
A  {\em relative $(4t+2,2,4t+2,2t+1)$-quasi-difference set} in
$E$ with forbidden subgroup $Z$ is a transversal 
$R$ for $Z$ in $E$ containing a subset 
$S\subset R\setminus \{1\}$ of size $2t+1$
such that, for all $x\in E\setminus Z$,}
\begin{equation}\label{condicionesquasidiffe}
 |R\cap xR| = 
 \left\{ 
 \begin{array}{ll}
 2t+1 & \hspace{15pt}  x\in  sZ 
 \ \mathrm{for} \ \mathrm{some} \ s\in S \\
  2t \ \mathrm{or} \ 2t+2 
  & \hspace{15pt}
 \mathrm{otherwise.} 
\end{array}
\right.
\end{equation}
\end{definition}

\vspace{4pt}

The familiar default assumption is that relative (quasi-) 
difference sets are normalized, i.e., contain $1$.
\begin{example}
$R=\{1,a,a^2,b,ab,a^2b\}$ is a relative 
$(6,2,6,3)$-quasi-difference set in 
$E=\langle a, b \; |\; a^3=b^2,\;b^4=1,\;a^b=a^5\rangle 
\cong Q_{12}$ with forbidden subgroup $Z=\langle a^3\rangle$.
\end{example}

It is clear from the definitions and Remark~\ref{mirror} that a
relative $(4t+2,2,4t+2,2t+1)$-quasi-difference set in $E$ 
is precisely a quasi-Hadamard subset of $E$. Together with
Theorem~\ref{orthococyquasihadamard}, we then have
\begin{proposition}\label{quasiort-quasidiff}
A cocycle $\psi\in Z^2( G, \langle-1\rangle)$ is quasi-orthogonal if 
and only if $\{(1,g)   \mid g\in G\}$ is a  relative 
$(4t+2,2,4t+2,2t+1)$-quasi-difference set in $E_\psi$ with 
forbidden subgroup $\langle (-1,1)\rangle$.
\end{proposition}

When $\psi$ is a coboundary, Proposition~\ref{RDSOC} gives an 
equivalence between
group-developed Hadamard matrices, Menon-Hadamard difference sets, 
and normal relative difference sets in $\mathbb{Z}_2\times G$
with forbidden subgroup $\mathbb{Z}_2\times \{ 1_G\}$; 
see \cite[Theorem~2.6, Corollary~2.7]{DFH00}.
This result has no counterpart in the context of
Proposition~\ref{quasiort-quasidiff}, 
since quasi-orthogonal coboundaries do not exist.

Suppose now that $k$ is not necessarily equal 
to $v$. The link between orthogonal cocycles 
and relative difference sets may be broadened in several ways. 
As shown in \cite{Gal04}, a relative  $(v,m,k,\lambda)$-difference 
set in $E$ with forbidden subgroup $N$ is equivalent to a
factor pair of $N$ by $G\cong E/N$ that is 
$(v,m,k,\lambda)$-orthogonal. 
The factor pair consists of a \emph{factor set} 
$\psi:G\times G\rightarrow N$ and a \emph{coupling} that 
together determine $E$; it is \emph{$(v,m,k,\lambda)$-orthogonal} 
with respect to a $k$-set $D\subseteq G$ if for each 
$x\in G\setminus \{1\}$ the sequence $\{ \psi(x, y) \}_{y\in D\cap x^{-1}D}$
is a listing of each element of $N$ exactly $\lambda$ times
(see \cite{Gal04} or \cite[Section 7.2]{Hor07}).
If $m= 2$ then the coupling is trivial and the set of 
factor pairs of $N$ by $G$ is just $Z^2(G,\mathbb{Z}_2)$. 
Moreover, an orthogonal cocycle is an orthogonal  
factor pair (with $k=v$ and $\lambda = v/2$).
The same is not true for quasi-orthogonal cocycles.
\begin{proposition}\label{nexc6}
There is no $(6,2,k,\lambda)$-orthogonal factor pair 
for any $k$, $\lambda>0$. Thus, none of the quasi-orthogonal 
cocycles over the groups of order $6$ is an orthogonal factor 
pair.
\end{proposition}
\begin{proof}
If a factor pair of 
$\mathbb{Z}_2 $ by $G$ is $(v,2,k,\lambda)$-orthogonal with 
respect to $D$ then $D$ is an ordinary $(v,k,2\lambda)$-difference 
set in $G$. But non-trivial $(6,k,\lambda)$-difference 
sets do not exist. \hfill $\Box$
\end{proof} 

\section{Partially balanced incomplete block designs}

A relative $(v,m,k,\lambda)$-difference set in $E$ with 
forbidden subgroup $N$ is equivalent to a divisible 
$(v,m,k,\lambda)$-design that is class regular with respect to 
$N$ and has $E$ as a regular group of 
automorphisms ($E$ acts  regularly on the points and blocks, 
while $N$ acts regularly on each of the $v$ point classes);
see \cite[Theorem~1.1.11, p.~13]{Pott}.
We establish the analogous passage 
between relative quasi-difference sets and 
partially balanced incomplete block designs.
A reference for the standard material in this section is  
 \cite[VI.1 and VI.42]{CRCHandbook}.

Let $X$ be a $v$-set and $R_0,R_1,\ldots,\abk R_m$ be 
nonempty subsets of $X \times X$, called 
\emph{associate classes}. The class $R_i$ is represented 
by an \emph{associate} (incidence) \emph{matrix}, 
i.e., a $(0,1)$-matrix $A_i$ indexed by $X$, with 
$1$ in row $x$ and column $y$ 
$\Leftrightarrow (x, y) \in R_i$.
The $R_i$s comprise an {\em association scheme on $X$} if
\begin{enumerate}
\item\label{Condit1} $A_0=I$
\item \label{Condit2}
$\sum^m_{i=0} A_i = J$ 

\vspace{2pt}

\item \label{Condit3} for all $i$, $A_i^\top= A_i$
\item \label{Condit4}
for all $i$, $j$ such that 
$i\leq j$, there are $p^k_{ij}\in \mathbb{N}$ such 
that $A_iA_j =\sum_k p^k_{ij}A_k$.
\end{enumerate}
Given such an association scheme, 
a {\em partially 
balanced incomplete  block design} $\mathrm{PBIBD}(m)$  
with parameters 
$v,b,r,k,\lambda_1, \abk \ldots , \lambda_m$ based
on $X$ has $b$ blocks, all of size $k$, 
each $x\in X$ occurs in exactly $r$ blocks, and if 
$(x,y)\in R_i$ then $x, y$ occur together in 
exactly $\lambda_i$ blocks.
\begin{theorem}[{\cite[42.4, pp.~562--563]{CRCHandbook}}]
\label{caracterizacionPBIBD}
Let $N$ be an incidence matrix of a $\mathrm{PBIBD}(m)$ 
with parameters $v,b,r,k,\lambda_1, \abk \ldots , \lambda_m$ 
corresponding to an association scheme with 
associate matrices $A_0, \ldots , A_m$. Then 
\begin{equation}\label{ASchemePBIBDSame}
NN^\top = rI + {\textstyle \sum}_{i=1}^m \lambda_iA_i \qquad 
\mbox{and} \qquad JN = kJ. 
\end{equation}
Conversely, a $v\times b$ $(0,1)$-matrix $N$ such that 
{\em (\ref{ASchemePBIBDSame})} holds
for associate matrices $A_i$ of an  association scheme
is an incidence matrix of 
a ${\rm PBIBD}(m)$ with parameters 
$v,b,r,k,\lambda_1,\ldots , \lambda_m$.
\end{theorem}

We now embark on the construction of a specific
$\mathrm{PBIBD}(4)$. Let $M$ be any $(-1,1)$-matrix 
satisfying (\ref{formachuli}) (so that 
if $M$ is cocyclic then the 
underlying cocycle is quasi-orthogonal). 
Form the expanded matrix
\[
{\footnotesize 
{\mathcal E}_M= \left[
\renewcommand{\arraycolsep}{.07cm}
\! \begin{array}{rr}
M & -M\\
-M & M
\end{array}\right]}.
\]
Put $A=\frac{1}{2}(J+M)$ and 
$\bar{A}=\frac{1}{2}(J-M)$; then the $(0,1)$-version of 
${\mathcal E}_{M}$ is
\begin{equation}\label{PhiMatrixDefinition}
\Phi={\footnotesize \left[
\renewcommand{\arraycolsep}{.12cm}
\begin{array}{rr}
A & \bar{A}\\
\bar{A} & A
\end{array}\right]}.
\end{equation}
Clearly
\begin{equation}\label{tamanobloques}
J\Phi=(4t+2) J.
\end{equation}
Next, we check that
\[
 AA^\top+\bar{A}\bar{A}^\top=
(4t+2)I+ (2t+2)\Delta_1+ 2t\Delta_2+(2t+1)\,
\left((J_2-I_2)\otimes J_{2t+1}\right),
\]
\[
\hspace{-65pt}
\bar{A}A^\top+A\bar{A}^\top= 
2t\Delta_1+ (2t+2)\Delta_2+(2t+1)\,(J_2-I_2)\otimes J_{2t+1}
\]
where
\[
 \ \ \ \ \ \ \
 \Delta_1= (\mathrm{Gr}(M)+2(I_2\otimes J_{2t+1})-(4t+4)I)/4,
\]
\[
\Delta_2=(2(I_2\otimes J_{2t+1})+4tI-\mathrm{Gr}(M))/4.
\]
Thus
\begin{equation}\label{Grammatrixincidence}
\Phi \Phi^\top=(4t+2)A_0+ 
(2t+1)A_2+(2t+2)A_3+2tA_4
\end{equation}
where
$A_0=I_{8t+4}$, $A_2=J_2\otimes (J_2-I_2)\otimes J_{2t+1}$,
$A_3=I_2\otimes \Delta_1+(J_2-I_2)\otimes \Delta_2$, and 
$A_4=I_2\otimes \Delta_2+(J_2-I_2)\otimes \Delta_1$.
Let $A_1=(J_2-I_2)\otimes I_{4t+2}$. Then
\begin{itemize}
\item[$\bullet$] $A_1^2=A_0,\, A_1A_2=A_2,\, 
A_1A_3=A_4,\,A_1A_4=A_3$.
\item[$\bullet$] $A_2^2=(4t+2)(A_0+A_1+A_3+A_4),\, 
A_2A_3=A_2A_4=2t A_2$.
\item[$\bullet$] $A_3^2=A_4^2=2tA_0+(2t-1)A_j,\, 
A_3A_4=2tA_1+(2t-1)A_{7-j}$ where $j\in \{ 3,4\}$.
\end{itemize}
So requirement~\ref{Condit4} in the definition of association 
scheme holds. Requirements~\ref{Condit1}--\ref{Condit3} hold 
as well. Therefore
\begin{lemma}\label{phiaesqu}
$A_0,A_1,A_2,A_3, A_4$ as above are the associate matrices of an
association scheme.
\end{lemma}

We now have our desired PBIBD.
\begin{proposition}\label{DesiredPBIBD4}
The matrix $\Phi$ as defined in {\em (\ref{PhiMatrixDefinition})}
for any $M$ satisfying 
{\em (\ref{formachuli})} is an incidence matrix of a 
${\rm PBIBD}(4)$ with parameters $v=b=8t+4$, $r=k=4t+2$,
$\lambda_1=0$, $\lambda_2=2t+1$, $\lambda_3=2t+2$, and 
$\lambda_4=2t$.
\end{proposition}
\begin{proof}
This follows from  (\ref{tamanobloques}), (\ref{Grammatrixincidence}),
Lemma~\ref{phiaesqu}, and Theorem~\ref{caracterizacionPBIBD}.
\hfill $\Box$
\end{proof}
\begin{example}\label{ejepbibd}
Let $t=1$ in Proposition~\ref{DesiredPBIBD4}. We choose
a quasi-orthogonal cocycle over $D_6$ 
whose matrix $A$ is visible in the top left quadrant of
\[
\Phi={\footnotesize 
\left[
\renewcommand{\arraycolsep}{.18cm}
\begin{array}{cccccccccccc}
 1& 1& 1& 1& 1& 1& 0& 0& 0& 0& 0& 0\\
 1& 0& 0& 1& 1& 1& 0& 1& 1& 0& 0& 0\\
 1& 0& 1& 0& 0& 0& 0& 1& 0& 1& 1& 1\\
 1& 1& 0& 0& 1& 0& 0& 0& 1& 1& 0& 1\\
 1& 1& 0& 0& 0& 1& 0& 0& 1& 1& 1& 0\\
 1& 1& 0& 1& 0& 0& 0& 0& 1& 0& 1& 1\\
 0& 0& 0& 0& 0& 0& 1& 1& 1& 1& 1& 1\\
 0& 1& 1& 0& 0& 0& 1& 0& 0& 1& 1& 1\\
 0& 1 & 0& 1& 1& 1& 1& 0& 1& 0& 0& 0\\
 0& 0& 1& 1& 0& 1& 1& 1& 0& 0& 1& 0\\
 0& 0& 1& 1& 1& 0& 1& 1& 0& 0& 0& 1\\
 0& 0& 1& 0& 1& 1& 1& 1& 0& 1& 0& 0
\end{array}\right]} .
\]
The non-trivial associate matrices are
\[
A_1 =
{\small \left[\renewcommand{\arraycolsep}{.14cm}
\begin{array}{cc} 
0_6 & I_6\\I_6 & 0_6 \end{array}\right]},
\ A_2=
{\footnotesize \left[\renewcommand{\arraycolsep}{.14cm}
\begin{array}{cccc} 0_3 & J_3 & 0_3 & J_3\\
J_3 & 0_3 & J_3 & 0_3 \\
0_3 & J_3 & 0_3 & J_3\\
J_3 & 0_3 & J_3 & 0_3  \end{array}\right]},
\
A_3=
{\small \left[\renewcommand{\arraycolsep}{.13cm}
\begin{array}{cc} 
\Delta_1 & \Delta_2\\ \Delta_2 & \Delta_1 
\end{array}\right]},
\
A_4=
{\small \left[\renewcommand{\arraycolsep}{.13cm}
\begin{array}{cc} 
\Delta_2 & \Delta_1\\ \Delta_1 & \Delta_2 
\end{array}\right]}
\]
where
\[
\Delta_1= 
{\footnotesize 
\left[\renewcommand{\arraycolsep}{.15cm}
\begin{array}{cccccc}
0 & 1 & 0 & 0 & 0 &0 \\
1 & 0 & 0 & 0 & 0 &0 \\
0 & 0 & 0 & 0 & 0 &0 \\
0 & 0 & 0 & 0 & 1 &1 \\
0 & 0 & 0 & 1 & 0 &1 \\
0 & 0 & 0 & 1 & 1 &0 \\
\end{array}\right]},\quad 
\Delta_2= 
{\footnotesize 
\left[\renewcommand{\arraycolsep}{.15cm}
\begin{array}{cccccc}
0 & 0 & 1 & 0 & 0 &0 \\
0 & 0 & 1 & 0 & 0 &0 \\
1 & 1 & 0 & 0 & 0 &0 \\
0 & 0 & 0 & 0 & 0 &0 \\
0 & 0 & 0 & 0 & 0 &0 \\
0 & 0 & 0 & 0 & 0 &0 \\
\end{array}\right]}.
\]
\end{example}

From now on, the notation $R_i$, $A_i$ is reserved for the 
association scheme of Lemma~\ref{phiaesqu}, and 
$\Phi$ is an incidence matrix of a corresponding
${\rm PBIBD}(4)$ 
with parameters $v=b=8t+4$, $r=k=4t+2$, 
$\lambda_1=0$, $\lambda_2=2t+1$, $\lambda_3=2t+2$,
$\lambda_4=2t$.

The next two theorems connect partially balanced incomplete 
block designs to quasi-orthogonal cocycles. 
\begin{theorem}\label{quasiorth-pbibd}
If $\psi \in Z^2(G, \langle-1\rangle)$ is 
quasi-orthogonal then $E_\psi$ is a regular 
group of automorphisms of the ${\rm PBIBD}(4)$ 
as in Proposition{\em ~\ref{DesiredPBIBD4}}.
The design is $R_1$-class regular with respect 
to $\langle (-1,1)\rangle$.
\end{theorem}
\begin{proof}
(Cf.~\cite[pp.~54--55]{DFH00}.) 
Choose any ordering $1, g_2,\ldots,g_{4t+2}$ of $G$, 
and index 
${\mathcal E}_{M_\psi}$ by $E=E_\psi$ under the ordering 
$(1,1), \ldots,\abk (1,g_{4t+2}),(-1,1), 
\ldots,(-1,g_{4t+2})$. Then
\[
{\mathcal E}_{M_\psi} =[ \phi(xy)]_{x,y\in E}
\]
where $\phi : (u,g)\mapsto u$. That is,
${\mathcal E}_{M_\psi}$ is group-developed over 
$E$. Hence $E$ acts as a regular
group of permutation automorphisms of ${\mathcal E}_{M_\psi}$; 
see \cite[Theorem~10.3.8, pp.~123--124]{DF11}.
Each of the point classes $\{( 1,g_i),(- 1,g_i)\}$ 
prescribed by $R_1$  is stabilized by 
$\langle (-1,1)\rangle$.  
\hfill $\Box$
\end{proof}

\begin{theorem}\label{Pbibd-rqds}
Suppose that a ${\rm PBIBD}(4)$ with incidence matrix
$\Phi$ has a central extension  
$E$ of $\langle -1\rangle$ as a regular group of 
automorphisms, and is $R_1$-class regular with respect 
to $\langle-1\rangle$. Then there exists a relative 
$(4t+2,2,4t+2,2t+1)$-quasi-difference set in $E$ with 
forbidden subgroup $\langle-1\rangle$.
\end{theorem}
\begin{proof}
By \cite[p.~15]{Pott} and
the hypothesis that $E$ is regular,
$\Phi^\top \Phi=\Phi \Phi^\top$. Thus $\Phi^\top$ is an
incidence matrix for a $\mathrm{PBIBD}(4)$ with the 
same parameters as those of $\Phi$. 
Index $\Phi$ by the elements
$x_1=1,\abk x_2,\ldots,\abk x_{8t+4}$ of $E$, where 
$x_i$ shifts column $1$ to column $i$. 
Note that $x_{4t+2+i}=-x_i$ because $\Phi$
is $R_1$-class regular with respect to  $\langle-1\rangle$. 
Let $R=\{x\in E \mid \Phi_{1,x}=1\}$.
Since $\lambda_1=0$, $R$ is a  transversal for 
 $\langle -1\rangle$ in $E$. 
Also $x^{-1}R = \{ y \in E \mid \Phi_{x,y}=1\}$; then
$|R\cap xR|= |R\cap x^{-1}R| =
(\Phi\Phi^\top)_{1,x}$ for any $x\in E$. 
Inspection of the first row of $\Phi\Phi^\top$ 
 reveals that $R$ and
$S=\abk \{x\in E \mid
(\Phi\Phi^\top )_{1,x}=2t+1 \ \mathrm{and} \ \Phi_{1,x}=1\}$
satisfy (\ref{condicionesquasidiffe}).
\hfill $\Box$
\end{proof}

\begin{remark}
Theorem~\ref{quasiorth-pbibd} and 
$\Phi^\top \Phi=\Phi \Phi^\top$ imply 
that if $\psi$ is quasi-orthogonal then 
$\mathrm{Gr}(M_\psi) = \mathrm{Gr}(M_{\psi}^{\top})$. 
Definition~\ref{REDefinition} may therefore
be framed equivalently in terms of column excess 
rather than row excess.
(However, note that the transpose of a cocyclic matrix 
indexed by a non-abelian group need 
not even be Hadamard equivalent to a cocyclic matrix.)
\end{remark}

Our final result should be compared with 
\cite[Theorem~2.4]{DFH00} and 
\cite[Corollary~7.31, p.~152]{Hor07}.
\begin{theorem}\label{mainequivalent} 
The following are equivalent.
\begin{itemize}
\item[{\rm I}.] $Z^2(G,\langle -1\rangle)$ contains a 
quasi-orthogonal cocycle.
\item[{\rm II}.] There is a 
relative $(4t+2,2,4t+2,2t+1)$-quasi-difference set with 
forbidden subgroup $\langle -1 \rangle$ in a quasi-Hadamard 
group $E$ such that $E/\langle-1\rangle\cong G$.
\item[{\rm III}.] There exists a ${\rm PBIBD}(4)$ 
with incidence matrix $\Phi$ on which a quasi-Hadamard group $E$ 
such that $E/\langle-1\rangle\cong G$ acts regularly, and which 
is $R_1$-class regular with respect to $\langle -1\rangle$.
\end{itemize}
\end{theorem}
\begin{proof}
We have ${\rm I}\Leftrightarrow {\rm II}$ by 
Theorem~\ref{orthococyquasihadamard} and 
Proposition~\ref{quasiort-quasidiff},
${\rm I}\Rightarrow {\rm III}$ by Theorem~\ref{quasiorth-pbibd}, 
and ${\rm III}\Rightarrow {\rm II}$ by 
Theorem~\ref{Pbibd-rqds}.  \hfill $\Box$
\end{proof}
\begin{remark}
The results cited in the proof 
of Theorem~\ref{mainequivalent}
enable us to explicitly construct each object
from any other equivalent object.
\end{remark}

\subsubsection*{Acknowledgments.} 
We thank Kristeen Cheng, Rob Craigen, Ronan Egan, and Kathy Horadam 
for helpful discussions. We are especially indebted to Eamonn O'Brien for 
his assistance with computations. The first author was supported by 
the project FQM-016 funded by JJAA (Spain). The second author was 
supported by the Irish Research Council New Foundations 
scheme (project `MatGroups').

\end{document}